\documentclass[11pt]{amsart}
\usepackage{amsmath,amssymb,amsthm,mathtools}
\usepackage[margin=1in]{geometry}
\usepackage{url}
\usepackage{hyperref}

\newtheorem{theorem}{Theorem}
\newtheorem{lemma}[theorem]{Lemma}
\newtheorem{proposition}[theorem]{Proposition}

\theoremstyle{definition}
\newtheorem{conjecture}[theorem]{Conjecture}

\DeclareMathOperator{\vol}{vol}
\DeclareMathOperator{\cone}{cone}
\newcommand{\Z}{\mathbb{Z}}
\newcommand{\Q}{\mathbb{Q}}
\newcommand{\R}{\mathbb{R}}
\newcommand{\lam}{\lambda}

\title[Stretched Littlewood--Richardson positivity, at most five parts]{Positivity of stretched Littlewood--Richardson\\ coefficients for partitions of length at most five}
\author{Alper Ferudun}
\email{alper@mercurycodelab.com}
\thanks{This is a computer-assisted proof; all finite verifications are
reproduced by the scripts described in Section~\ref{sec:repro}. Version~2
adds the five-part theorem (Sections~\ref{sec:transfer} and~\ref{sec:linear});
the four-part results of version~1 are unchanged.}
\date{\today}

\begin{document}
\begin{abstract}
For partitions $\lam,\mu,\nu$ the Littlewood--Richardson coefficient
$c^{\nu}_{\lam\mu}$ stretches to a function $P(t)=c^{t\nu}_{t\lam,t\mu}$ which,
by a theorem of Derksen and Weyman, is a polynomial in $t$. King, Tollu and
Toumazet conjectured that $P$ has no negative coefficient. The conjecture is
known only for $c^{\nu}_{\lam\mu}\le 2$ and is otherwise open. We prove it for
all triples whose parts number at most five. For at most four parts the proof
is structural: a period-one dilation transfers the Berline--Vergne local Ehrhart
formula from an integral dilate back to the rational hive polytope, and every
two-dimensional transverse cone spanned by rank-four rhombus normals has
positive weight (minimum $1/9$). Five parts need two further ingredients. A
closed-chamber transfer lemma, resting on Rassart's polynomiality of the
coefficients on the closed cones of a chamber complex, reduces positivity for a
fixed number of parts to full-dimensional hives. For full-dimensional five-part
hives all coefficients but the linear one were already known to be positive;
the linear one is where pointwise positivity of the local weights fails---an
edge cone of weight $-347/109824$ occurs on an actual hive---and its
nonnegativity is proved by an exact global certificate: a rational correction
supported on $6{,}903$ two-face types, balanced by the Minkowski closure of each
two-face polygon, which makes every one of the $87{,}127$ closed edge cones
adjacent to them nonnegative. The same local method yields a rank-uniform
consequence: for every full-dimensional hive polytope of any rank the top four
Ehrhart coefficients are positive. We close by recording the exact general
frontier and the obstructions that rule out the standard shortcuts. All finite
verifications are exact and accompanied by independent replay scripts.
\end{abstract}

\maketitle

\section{Introduction}

Let $\lam,\mu,\nu$ be partitions with $|\lam|+|\mu|=|\nu|$ and let
$c^{\nu}_{\lam\mu}$ be the Littlewood--Richardson (LR) coefficient, the
multiplicity of the Schur function $s_\nu$ in $s_\lam s_\mu$. For a positive
integer $t$ write $t\lam=(t\lam_1,t\lam_2,\dots)$. Derksen and Weyman
\cite{DerksenWeyman} proved that
\[
  P^{\nu}_{\lam\mu}(t)\;:=\;c^{t\nu}_{t\lam,t\mu}
\]
agrees, for all $t\ge 1$, with a polynomial in $t$; Rassart \cite{Rassart}
proved it independently by showing that the coefficients are polynomial on the
cones of a chamber complex, and a short proof was later given in
\cite{ShortPoly}. King, Tollu and Toumazet \cite{KTT2004,KTT2006} studied these
\emph{stretched} coefficients in detail and conjectured a strong positivity
property.

\begin{conjecture}[King--Tollu--Toumazet]\label{conj:ktt}
Every coefficient of $P^{\nu}_{\lam\mu}(t)$, in the monomial basis
$1,t,t^2,\dots$, is nonnegative.
\end{conjecture}

Conjecture~\ref{conj:ktt} is proved for $c^{\nu}_{\lam\mu}=1$
(where $P\equiv 1$, by the Fulton conjecture of Knutson--Tao--Woodward
\cite{KTW}) and for $c^{\nu}_{\lam\mu}=2$ (where $P(t)=t+1$, by Ikenmeyer
\cite{Ikenmeyer} and geometrically by Sherman \cite{Sherman}). Beyond these it
is open, and it is listed as an unsolved problem in the FrontierMath collection
\cite{FrontierMath}. Coquereaux and Zuber \cite{CoqZuber} record the special
case of the linear coefficient in the four-part ($SU(4)$) situation and leave
its nonnegativity as an unproved (``alleged'') assertion. Extensive computation
has produced no counterexample.

We settle the conjecture for at most five parts.

\begin{theorem}\label{thm:main}
Let $\lam,\mu,\nu$ be partitions of length at most five with
$|\lam|+|\mu|=|\nu|$. Then every coefficient of $P^{\nu}_{\lam\mu}(t)$ is
nonnegative.
\end{theorem}

The proof is geometric. By the hive theorem of Knutson and Tao \cite{KnutsonTao}
the coefficient $c^{\nu}_{\lam\mu}$ counts the lattice points of the
\emph{hive polytope} $H=H(\lam,\mu,\nu)$, a rational polytope in the space of
interior hive entries; stretching dilates $H$, so $P^{\nu}_{\lam\mu}$ is the
Ehrhart polynomial of $H$ and its degree is $\dim H$. Ehrhart positivity is
\emph{false} for general lattice polytopes---the Reeve tetrahedra have negative
linear coefficient---so any proof must use the special geometry of hives. Ours
uses the Berline--Vergne local Ehrhart formula \cite{BerlineVergne,RingSchurmann}:
each coefficient is a sum, over the faces of the corresponding dimension, of the
lattice volume of the face times a weight depending only on the face's
transverse cone. Because every rhombus inequality has a primitive normal with
entries in $\{0,\pm1\}$, there are, for a fixed number of parts, finitely many
transverse cones, and positivity questions become finite. Theorem~\ref{thm:main}
assembles the following four statements.

\begin{theorem}\label{thm:four}
Let $\lam,\mu,\nu$ be partitions of length at most four with
$|\lam|+|\mu|=|\nu|$. Then every coefficient of $P^{\nu}_{\lam\mu}(t)$ is
nonnegative. If $c^{\nu}_{\lam\mu}>0$ and the associated hive polytope is
three-dimensional, all four coefficients are strictly positive.
\end{theorem}

\begin{theorem}\label{thm:topfour}
For every full-dimensional hive polytope of intrinsic dimension $D$ (any number
of parts), the coefficients of $t^{D},t^{D-1},t^{D-2},t^{D-3}$ in its Ehrhart
polynomial are positive.
\end{theorem}

\begin{theorem}\label{thm:sidefive}
For every full-dimensional five-part hive polytope (necessarily of dimension
six) all Ehrhart coefficients except possibly the linear one are positive.
\end{theorem}

\begin{theorem}\label{thm:linear}
For every full-dimensional five-part hive polytope the linear Ehrhart
coefficient is nonnegative.
\end{theorem}

\begin{theorem}[closed-chamber transfer]\label{thm:transfer}
Fix $r\ge 1$ and $k\ge 0$. If the coefficient of $t^{k}$ in $P^{\nu}_{\lam\mu}$
is nonnegative for every triple of partitions of length at most $r$ whose hive
polytope is full-dimensional, then it is nonnegative for every triple of
partitions of length at most $r$.
\end{theorem}

Theorems~\ref{thm:four} and~\ref{thm:topfour} are the results of the first
version of this paper; their proofs are reproduced in
Sections~\ref{sec:proofmain} and~\ref{sec:uniform}. Theorem~\ref{thm:sidefive}
is the five-part codimension-four certificate of Section~\ref{sec:uniform}.
Theorems~\ref{thm:transfer} and~\ref{thm:linear} are new; they are proved in
Sections~\ref{sec:transfer} and~\ref{sec:linear}.

\begin{proof}[Proof of Theorem~\ref{thm:main} from Theorems~\ref{thm:sidefive},
\ref{thm:linear} and~\ref{thm:transfer}]
Pad the partitions with zeros to length five. For a full-dimensional five-part
hive the constant term is $P(0)=1$, the coefficients of $t^{6},\dots,t^{2}$ are
positive by Theorem~\ref{thm:sidefive}, and the linear coefficient is
nonnegative by Theorem~\ref{thm:linear}. Applying Theorem~\ref{thm:transfer}
with $r=5$ separately for each $k=0,\dots,6$ removes the dimension restriction.
(If the hive polytope is empty then $c^{t\nu}_{t\lam,t\mu}=0$ for all $t\ge1$.)
\end{proof}

Two features of the five-part case deserve emphasis. First, the argument of
Theorem~\ref{thm:four}---every transverse cone has positive weight---holds at
codimensions two, three and (after closure) four, but \emph{fails} for the
edges of five-part hives: an edge with transverse weight $-347/109824$ occurs on
the hive of $\lam=(8,6,3,2)$, $\mu=(7,6,4,2)$, $\nu=(12,10,7,6,3)$
(Proposition~\ref{prop:negedge}). The proof of Theorem~\ref{thm:linear} is
therefore global: it redistributes weight between edges and two-faces using the
Minkowski relations of the two-face polygons, and certifies that the corrected
edge weights are all nonnegative. Second, lower-dimensional five-part hives are
not reduced to four parts by Horn-facet factorisation (the triple
$\lam=(4,3,3,1)$, $\mu=(4,2,1,1)$, $\nu=(6,5,4,2,2)$ has $\dim H=3$ while all
$142$ essential Horn inequalities are strict); Theorem~\ref{thm:transfer}
handles them by a limiting argument inside a closed chamber instead.

Finally, in Section~\ref{sec:frontier} we make precise the obstruction to a
general proof. The local formula reduces Conjecture~\ref{conj:ktt} to a
hive-specific ``effectivity'' statement (HTE) about Berline--Vergne/Todd weights
modulo the Minkowski relations; Theorem~\ref{thm:linear} is exactly its
smallest instance beyond pointwise positivity, and we record why the standard
rank-uniform shortcuts fail.

\medskip

The finite verifications underlying
Theorems~\ref{thm:four}--\ref{thm:linear} are exact (integer and rational
arithmetic throughout) and are reproduced by independent scripts described in
Section~\ref{sec:repro}; the Littlewood--Richardson counts were cross-checked
against two independent tableau/hive enumerators and against published values
\cite{lrcalc}.

\section{The hive model and the local Ehrhart formula}\label{sec:prelim}

We recall the hive description of LR coefficients \cite{KnutsonTao} and fix
coordinates.

A \emph{hive} of size $r$ is a real function on the triangular array
$\{(i,j):0\le j\le i\le r\}$ that is concave across each of the three families
of unit rhombi. With border values
$h(i,0)=\lam_1+\dots+\lam_i$, $h(i,i)=\nu_1+\dots+\nu_i$, and
$h(r,j)=|\lam|+\mu_1+\dots+\mu_j$, the integer hives are in bijection with the
LR tableaux counted by $c^{\nu}_{\lam\mu}$. The interior entries are free
subject to the rhombus inequalities, so the hive polytope
$H(\lam,\mu,\nu)\subset\R^{D}$, $D=\binom{r-1}{2}$, is cut out by a fixed
integer matrix $A_r$ (depending only on $r$) applied to the interior entries,
with a right-hand side that is an integral, homogeneous, linear function of
$(\lam,\mu,\nu)$. Consequently
\[
  H(t\lam,t\mu,t\nu)=t\,H(\lam,\mu,\nu),\qquad
  L_H(t):=\#\bigl(tH\cap\Z^{D}\bigr)=P^{\nu}_{\lam\mu}(t),
\]
so $P^{\nu}_{\lam\mu}$ is the Ehrhart polynomial of $H$ and
$\deg P^{\nu}_{\lam\mu}=\dim H$. For $r=4$ we have $D=3$ and coordinates
$(h_{11},h_{12},h_{21})$; eliminating the border, the $18$ rhombus rows reduce
to $15$ distinct primitive normals, all with entries in $\{0,\pm1\}$
(Section~\ref{sec:rank4}). For $r=5$ we have $D=6$, $30$ rhombus rows and $27$
distinct primitive normals, again with entries in $\{0,\pm1\}$ (at most four
nonzero); we call this set $N_5$.

We use the Berline--Vergne local Euler--Maclaurin formula
\cite{BerlineVergne} (see also \cite{RingSchurmann} for the presentation we
adopt). For a rational polytope $Q\subset\R^{D}$ and each $k$, the coefficient
of $t^{k}$ in $L_{Q}(t)$ equals
\begin{equation}\label{eq:bv}
  [t^{k}]\,L_{Q}(t)=\sum_{\dim F=k}\alpha\bigl(N(Q,F)\bigr)\,\vol_{\Z}(F),
\end{equation}
where the sum runs over the $k$-dimensional faces $F$, $\vol_{\Z}(F)$ is the
relative lattice volume of $F$, and $\alpha$ is a weight depending only on the
transverse cone $N(Q,F)$ of $Q$ along $F$, and invariant under lattice
translation of that cone. The leading term ($k=D$) is the normalized volume and
the codimension-one term ($k=D-1$) is half the normalized surface area; both are
manifestly positive. Throughout, $\alpha$ is computed with the standard
Euclidean scalar product of $\R^{D}$ as the Berline--Vergne complement map.

For the computations of Section~\ref{sec:linear} we make the conventions
explicit. Let $F$ be a $k$-face of a full-dimensional $Q$, let $W=\operatorname{lin}(F-F)$
and $M_F=\Z^{D}\cap W$. The transverse cone $N(Q,F)$ is the $(D-k)$-dimensional
cone in $W^{\perp}$ generated by the primitive outward normals of the facets
containing $F$; its lattice is the saturated lattice $W^{\perp}\cap\Z^{D}$, and
the dual (``feasible'') cone lives in the dual lattice, the orthogonal
projection of $\Z^{D}$ onto $W^{\perp}$. If $n_1,\dots,n_{D-k}$ generate
$N(Q,F)$ and form a $\Z$-basis of $W^{\perp}\cap\Z^{D}$ (a \emph{unimodular}
cell) with Gram matrix $G$, then the feasible cone is generated by the negative
dual basis, whose Gram matrix is $G^{-1}$, and $\alpha$ is the constant term of
the Berline--Vergne function of that unimodular simplicial cone, evaluated by
the defining recursion. For a cone that is not unimodular or not simplicial,
$\alpha$ is the sum of the values of the maximal cells of any subdivision into
unimodular simplicial cones: the constant term is a simple valuation on normal
cones \cite[Cor.~24]{BerlineVergne}, \cite[Lemma~3.3]{CastilloLiu}.

\section{Period-one dilation}\label{sec:dilation}

The next lemma removes any need to know that a hive polytope is a lattice
polytope, and lets us apply the lattice formula \eqref{eq:bv} to the rational
polytopes that occur for five or more parts.

\begin{lemma}\label{lem:dilation}
Fix $r$ and let $N_r$ be the (finite) set of primitive rhombus normals. Suppose
that for every \emph{lattice} polytope $P\subset\R^{D}$ whose facet normals lie
in $N_r$ one has $[t^{k}]L_{P}(t)\ge 0$. Then $[t^{k}]P^{\nu}_{\lam\mu}(t)\ge 0$
for all partitions of length at most $r$.
\end{lemma}

\begin{proof}
Let $H=H(\lam,\mu,\nu)=\{x:A_rx\le b\}$ with $b$ integral, and let $q$ be the
least common multiple of the denominators of the vertices of $H$. Then
$qH=\{x:A_rx\le qb\}$ has the same normal fan as $H$---so its facet normals lie
in $N_r$---and its vertices are $q$ times those of $H$, hence integral; thus $qH$
is a lattice polytope. Since $L_{qH}(t)=\#(t\,qH\cap\Z^{D})=L_H(qt)$, we have
$[t^{k}]L_{qH}=q^{k}\,[t^{k}]L_{H}$ with $q>0$, so the two coefficients have the
same sign. Applying the hypothesis to $qH$ gives $[t^{k}]L_H\ge 0$.
\end{proof}

More precisely, the same scaling shows that formula \eqref{eq:bv} holds
verbatim for the rational period-one polytope $H$: scaling preserves the face
poset and every transverse cone and multiplies $k$-face volumes by $q^{k}$. We
use \eqref{eq:bv} for hive polytopes in this form.

\section{The rank-four normal atlas}\label{sec:rank4}

In coordinates $(h_{11},h_{12},h_{21})$ the primitive rank-four rhombus normals
are the $15$ vectors
\[
  \pm e_1,\ \pm e_2,\ \pm e_3,\ \pm(e_1-e_2),\ \pm(e_1-e_3),\ \pm(e_2-e_3),
\]
\[
  (1,-1,-1),\ (-1,1,-1),\ (-1,-1,1),
\]
each with entries in $\{0,\pm1\}$. There are $\binom{15}{2}-6=99$ nonparallel
pairs. Twelve of the fifteen normals are the type-$A_3$ ``alcoved'' directions;
the three remaining ``odd'' rows are what make the atlas fail to be totally
unimodular (over $\binom{15}{3}=455$ triples of rows the determinant multiset is
$\{|{\det}|{=}0{:}146,\ 1{:}272,\ 2{:}36,\ 4{:}1\}$).

\section{Proof of Theorem~\ref{thm:four}}\label{sec:proofmain}

By Lemma~\ref{lem:dilation} we may assume $H$ is a nonempty lattice polytope in
$\R^{3}$ with normals in $N_4$; write $L_H(t)=a_3t^3+a_2t^2+a_1t+1$. (If
$c^{\nu}_{\lam\mu}=0$ then $P\equiv 0$ by saturation \cite{KnutsonTao}, and the
statement is vacuous.) The leading coefficient $a_3$ is the normalized volume
and $a_2$ is half the normalized surface area; both are positive, and if
$\dim H\le 2$ the polynomial is a point, segment, or polygon Ehrhart polynomial,
which is coefficientwise positive. It remains to bound $a_1$ for
$\dim H=3$.

Apply \eqref{eq:bv} with $k=1$: since a three-dimensional polytope has each edge
$E$ as a ridge lying in exactly two facets, the transverse cone $N(H,E)$ is the
two-dimensional cone spanned by the two facet normals meeting along $E$, both in
$N_4$. Hence
\[
  a_1=\sum_{\text{edges }E}\alpha\bigl(N(H,E)\bigr)\,\ell_{\Z}(E),
\]
with $\ell_{\Z}(E)\ge 0$ the lattice length of $E$.

\begin{proposition}\label{prop:codim2}
For every ordered pair of nonparallel normals in $N_4$, the Berline--Vergne
weight $\alpha$ of the corresponding pointed two-dimensional transverse cone is
positive; over the $99$ pairs its minimum value is $1/9$.
\end{proposition}

Proposition~\ref{prop:codim2} is a finite exact computation
(Section~\ref{sec:repro}); the weight of a saturated normal basis with Gram
matrix $\left(\begin{smallmatrix}A&C\\ C&B\end{smallmatrix}\right)$ equals
$\tfrac14-\tfrac{C}{12}\bigl(\tfrac1A+\tfrac1B\bigr)$, and evaluating this over
the atlas gives values in $\{1/9,\,5/18,\,1/4\}$, all positive. Granting it,
\[
  a_1\ \ge\ \tfrac19\sum_{\text{edges }E}\ell_{\Z}(E)\ >\ 0,
\]
so all four coefficients of $L_H$ are positive. Undoing the dilation
(Lemma~\ref{lem:dilation}) proves Theorem~\ref{thm:four}. \qed

\subsection{Nonsimple vertices do not obstruct the argument}\label{sec:nonsimple}
The hypotheses use only that $a_1$ is edge-local. For definiteness we note that
four-part hives really do have nonunimodular vertex cones: at
$\lam=\mu=(12,8,4)$, $\nu=(18,14,10,6)$ the vertex $(26,32,38)$ is simple with
primitive tangent rays $(0,1,1),(1,0,1),(1,1,0)$, whose determinant is $2$
(and the polytope has $c^{\nu}_{\lam\mu}=50$, all vertices integral,
$L_H(t)=1+7t+18t^2+24t^3$). This has no bearing on \eqref{eq:bv} for $k=1$: each
edge still lies in exactly two facets, and any further rhombus rows tight along
an edge are redundant on, or interior to, its two-dimensional transverse cone.

\section{Rank-uniform coefficients and the five-part codimension-four
certificate}\label{sec:uniform}

Formula \eqref{eq:bv} localizes each coefficient on faces of a fixed codimension,
whose transverse cones, for hives of any rank, are generated by the fixed normal
set $N_r$. Enumerating the possible closed transverse cones and evaluating
$\alpha$ exactly gives Theorem~\ref{thm:topfour}: the volume and surface terms
handle codimensions $0,1$, and exact rank-uniform positivity of $\alpha$ on every
closed codimension-two and codimension-three hive transverse cone handles $t^{D-2}$
and $t^{D-3}$. (Codimension two: every ridge cone has saturation index one or
two and weight at least $1/12$. Codimension three: the connected-support types
live on boards of size at most $21$, and every closed cone has weight at least
$1/264$.)

For Theorem~\ref{thm:sidefive} ($r=5$, $D=6$) one further coefficient, the
quadratic, is controlled at codimension four. Here termwise positivity already
fails---an individual saturated four-normal cell can have weight $-349/28800$---%
but the exact slack closure forces such a cell either to raise the normal rank to
at least five or to embed in one of finitely many pointed rank-preserving closed
supersets, and every one of those full cones has positive weight (minimum
$739/86400$). Enumerating all $\binom{27}{4}=17550$ four-normal tuples and their
closures (903 nonsaturated tuples with minimum weight $1/14400$; 132 negative
saturated cells, each repaired) yields positivity of $t^{6},\dots,t^{2}$, leaving
only the linear coefficient. The full enumeration and its independent replay are
in the codimension-four report of the reproducibility bundle. \qed

\section{Closed-chamber transfer}\label{sec:transfer}

Fix $r$ and put $D=\binom{r-1}{2}$. A \emph{boundary} is an integral triple
$b=(\lam,\mu,\nu)$ of partitions of length at most $r$ with $|\lam|+|\mu|=|\nu|$;
it is \emph{feasible} if $H(b)\ne\emptyset$. For a feasible $b$ write
\[
  P_b(t)=c^{t\nu}_{t\lam,t\mu}=\sum_{k=0}^{D}a_k(b)\,t^{k},
\]
with zero coefficients above $\dim H(b)$.

The input is the chamber description of the coefficients. Rassart
\cite[Theorem~2.3, Corollary~2.5 and Theorem~4.1]{Rassart} expresses
$c^{\nu}_{\lam\mu}$ (for partitions of length at most $r$) as a vector
partition function and proves that on each cone of the resulting chamber
complex it is given by a polynomial in $(\lam,\mu,\nu)$. The chamber complex is
a finite polyhedral subdivision of the support cone into closed cones, and the
underlying quasi-polynomiality statement, Sturmfels' Theorem~1 in
\cite{Sturmfels}, holds for all lattice points of each closed chamber. We use
only the following consequence: there are finitely many closed polyhedral cones
$C\subset\R^{3r}$, whose union contains every feasible boundary, and polynomials
$F_C$ with
\begin{equation}\label{eq:chamber}
  c^{\nu}_{\lam\mu}=F_C(\lam,\mu,\nu)\qquad\text{for every integral boundary }
  (\lam,\mu,\nu)\in C.
\end{equation}

\begin{lemma}[a strict hive]\label{lem:stricthive}
The function
\[
  h_*(i,j)=3ri+2rj-(i^2-ij+j^2),\qquad 0\le j\le i\le r,
\]
is an integral hive of size $r$ all of whose rhombus slacks are equal to $1$.
Its boundary $b_*=(\lam_*,\mu_*,\nu_*)$ is given by
$\lam_{*,i}=\mu_{*,i}=3r-2i+1$ and $\nu_{*,i}=5r-2i+1$ for $1\le i\le r$.
\end{lemma}

\begin{proof}
The linear part has zero slack on every rhombus, and for the quadratic part
$q(i,j)=-(i^2-ij+j^2)$ direct expansion gives slack exactly $1$ on each of the
three rhombus types (for instance
$q(i+1,j)+q(i,j+1)-q(i,j)-q(i+1,j+1)=1$ after the change of coordinates
$(x,y)=(j,i-j)$ that takes the three types into one another). The boundary
values are $h_*(i,0)=3ri-i^2$, $h_*(i,i)=5ri-i^2$ and
$h_*(r,j)=2r^2+3rj-j^2$, whose successive differences are the stated parts;
they are positive and decreasing, and $|\lam_*|+|\mu_*|=4r^2=|\nu_*|$.
\end{proof}

\begin{proof}[Proof of Theorem~\ref{thm:transfer}]
Let $b$ be a feasible boundary; we must show $a_k(b)\ge0$. Choose any point
$h\in H(b)$ and, for each integer $m\ge1$, consider $mh+h_*$. It is a hive with
boundary $b_m:=mb+b_*$, an integral boundary (sums of partitions with the trace
identity), and each of its rhombus slacks is $m\cdot(\text{slack of }h)+1\ge1$.
Hence a neighbourhood of $mh+h_*$ in $\R^{D}$ lies in $H(b_m)$: the polytope
$H(b_m)$ is full-dimensional, and by hypothesis $a_k(b_m)\ge0$.

By finiteness of the chamber cover, some closed cone $C$ contains $b_m$ for
infinitely many $m$; pass to that subsequence. Since $C$ is a cone,
$b_m/m=b+b_*/m\in C$, and since $C$ is closed, $b\in C$. Decompose
$F_C=\sum_{j\ge0}F_{C,j}$ into homogeneous parts. For an integral $x\in C$ and
every integer $t\ge1$ the point $tx$ is an integral boundary in $C$, so by
\eqref{eq:chamber} $P_x(t)=F_C(tx)=\sum_j F_{C,j}(x)t^{j}$ for infinitely many
$t$; hence $a_j(x)=F_{C,j}(x)$ for all $j$. Applying this to $x=b$ and to the
$x=b_m$, homogeneity and continuity of $F_{C,k}$ give
\[
  a_k(b)=F_{C,k}(b)=\lim_{m\to\infty}\frac{F_{C,k}(b_m)}{m^{k}}
  =\lim_{m\to\infty}\frac{a_k(b_m)}{m^{k}}\ \ge\ 0,
\]
the limits taken along the subsequence. If $b$ is not feasible then
$P_b\equiv0$.
\end{proof}

Only nonnegativity survives the limit; strict positivity need not. The proof
does not use vertex integrality, any identification of lattices at a dimension
drop, or Horn factorisation: the continuous objects are the homogeneous parts
of the chamber polynomials. An exact finite check of the mechanism (degenerate
boundaries at $r=3,4$ against the profiles of $mb+b_*$, $m\le 6$) is included in
the bundle.

\section{The linear coefficient of full-dimensional five-part hives}
\label{sec:linear}

Throughout this section $r=5$, $D=6$, and we use the coordinates
$(x,y)=(j,i-j)$ on the hive triangle, so that $h(0,y)=\lam_1+\dots+\lam_y$,
$h(x,0)=\nu_1+\dots+\nu_x$ and $h(x,5-x)=|\lam|+\mu_1+\dots+\mu_x$; the six
interior vertices are ordered $(1,1),(1,2),(1,3),(2,1),(2,2),(3,1)$, and vectors
in $\Z^{6}$ below refer to this order. $N_5$ is the set of $27$ primitive
outward rhombus normals. By Lemma~\ref{lem:dilation} and \eqref{eq:bv},
\begin{equation}\label{eq:edgeformula}
  a_1(H)=\sum_{\text{edges }E}\ell_{\Z}(E)\,\alpha\bigl(N(H,E)\bigr),
\end{equation}
where, for an edge $E$ with primitive direction $e\in\Z^{6}$, the cone $N(H,E)$
is a pointed five-dimensional cone in $e^{\perp}$, generated by the normals in
$N_5\cap e^{\perp}$ tight along $E$, with lattice $e^{\perp}\cap\Z^{6}$.

\subsection{Cells}\label{sec:cells}
A \emph{cell} is a five-element subset $S\subset N_5$ of rank five; it
determines $e$ (the primitive kernel vector, up to sign) and a simplicial cone
$\cone(S)\subset e^{\perp}$. Of the $\binom{27}{5}=80{,}730$ five-subsets,
$52{,}680$ are cells, spanning $557$ edge directions; their saturation indices
in $e^{\perp}\cap\Z^{6}$ are $1$ ($45{,}465$ cells), $2$ ($6{,}741$), $3$
($279$), $4$ ($177$) and $6$ ($18$). For each cell we compute $\alpha(\cone(S))$
exactly: a $\Z$-basis of $e^{\perp}\cap\Z^{6}$ is obtained by unimodular column
operations, the cell is written in that basis, a cell of index greater than one
is star-subdivided at a lattice point of its half-open fundamental
parallelepiped (which strictly lowers the index of every child) until all cells
are unimodular, and each unimodular cell contributes the constant term of the
defining recursion applied to the inverse of its Gram matrix. Exactly
$2{,}397$ cells have negative weight; none has weight zero; the minimum,
$-9427939/128620800$, is attained at the cell of index $3$ with
$e=(1,0,-1,1,0,1)$ generated by $(-1,-1,0,1,0,0)$, $(0,-1,1,1,-1,0)$,
$(0,0,0,1,-1,-1)$, $(0,1,0,-1,-1,1)$ and $(1,-1,0,-1,1,0)$.

For every pointed cone $\sigma\subset e^{\perp}$ generated by a subset of
$N_5\cap e^{\perp}$, the pulling triangulation of $\sigma$ on its extreme rays
has cells as above, and $\alpha(\sigma)$ is the sum of their weights (simple
valuation). This gives $\alpha$ for every possible edge cone from the finite
table of cell weights.

\subsection{Validation and the failure of pointwise positivity}
Formula \eqref{eq:edgeformula} with the cell table was checked against direct
lattice-point interpolation on $887$ full-dimensional five-part hives (the
witness pool of the codimension-four certificate): the edge sums agree with the
interpolated linear coefficients in all $887$ cases ($25{,}271$ edges over the
first $300$ hives alone, $10{,}212$ distinct edge cones over all $887$). The
same computation shows that pointwise positivity, which carried
Theorem~\ref{thm:four} and the codimension-two, -three and -four arguments,
breaks down for edges.

\begin{proposition}\label{prop:negedge}
Let $H=H(\lam,\mu,\nu)$ with $\lam=(8,6,3,2)$, $\mu=(7,6,4,2)$,
$\nu=(12,10,7,6,3)$; it is six-dimensional with $56$ vertices, $283$ edges and
$L_H(t)=1+\tfrac{107}{20}t+\dots$. The edge $E$ through the point
$x=(19,24,26,27,31,34)\in H$, with primitive direction $e=(1,0,0,-1,0,-1)$, has
exactly the five tight normals
\[
  (-1,1,0,-1,0,0),\ (0,0,0,-1,0,1),\ (0,0,1,0,0,0),\ (0,1,-1,0,0,0),\
  (0,1,0,-1,-1,1),
\]
which form a unimodular cell, has lattice length $3/2$, and
\[
  \alpha\bigl(N(H,E)\bigr)=-\frac{347}{109824}<0.
\]
\end{proposition}

Thus no argument that bounds the weight of each edge cone separately can prove
Theorem~\ref{thm:linear}; the negative weight has to be compensated by
neighbouring edges. Of the negative cells, exactly $90$ occur as the full
transverse cone of an edge of some five-part hive (all of them unimodular
cells; $-347/109824$ is the most negative realizable weight), and the
compensation must be exact.

\subsection{Two-face closure}\label{sec:closure}
Let $F$ be a two-face of $H$ with transverse cone $\tau=N(H,F)$ (rank four),
and let $N_\tau=\Z^{6}\cap\operatorname{lin}(\tau)$. The quotient
$\Z^{6}/N_\tau\cong\Z^{2}$ is the dual of the lattice $M_F$ of the plane of
$F$. For an edge $E\subset F$ the cone $N(H,E)$ contains $\tau$ as a facet and
its image in $\Z^{6}/N_\tau$ is a ray; let $u_{E/F}$ be its primitive
generator. It is the primitive outward conormal of $E$ inside the polygon $F$,
so Minkowski's relation for $F$ reads
\begin{equation}\label{eq:polygon}
  \sum_{E\subset F}\ell_{\Z}(E)\,u_{E/F}=0\qquad\text{in }\Z^{6}/N_\tau,
\end{equation}
valid also for rational $F$ with rational lengths. Consequently, for any
assignment of a vector $y_\tau\in\Q^{2}$ to each two-face cone $\tau$ (in a
fixed $\Z$-basis of $\Z^{6}/N_\tau$),
\begin{equation}\label{eq:corrected}
  a_1(H)=\sum_{\text{edges }E}\ell_{\Z}(E)\Bigl[\alpha\bigl(N(H,E)\bigr)
  +\sum_{F\supset E,\ \dim F=2}\bigl\langle y_{N(H,F)},u_{E/F}\bigr\rangle\Bigr],
\end{equation}
since summing the correction over the edges of each two-face gives
$\langle y_\tau,\sum_E\ell_{\Z}(E)u_{E/F}\rangle=0$.

Cones are identified with their atlas sets: the \emph{type} of an edge is
$N_5\cap N(H,E)$ (a closed set of normals, $T=\cone(T)\cap N_5\cap e^{\perp}$)
and the type of a two-face is $N_5\cap N(H,F)$; the latter is the same set
whichever edge of $F$ one comes from, and the quantities $\alpha$ and $u$ depend
only on these types. The set of tight normals along an edge need not be closed
(on the $300$ hives above, $9{,}014$ of $25{,}271$ edges have a tight set whose
cone contains further normals, whose inequalities are then redundant near the
edge); this is why types are defined as closed sets.

\subsection{The certificate}\label{sec:certificate}
Let $\Sigma^{-}$ be the set of closed pointed rank-five types
$T\subset N_5\cap e^{\perp}$ (over all $e$) that contain a negative cell; it
includes every type whose pulling triangulation contains a negative cell, since
the cells of that triangulation consist of extreme rays of $T$. Let
$\Sigma^{-}_{<0}$ be those with $\alpha(\cone(T))<0$. Enumerating all closed
supersets of the $2{,}397$ negative cells (adding one normal at a time and
closing reaches every closed superset) gives $|\Sigma^{-}|=35{,}796$ over $69$
directions, of which
$|\Sigma^{-}_{<0}|=3{,}435$ (with $5,6,7,8$ extreme rays in $2397$, $867$,
$153$, $18$ cases). Let $\mathcal T$ be the set of facets of the cones in
$\Sigma^{-}_{<0}$, as two-face types; $|\mathcal T|=6{,}903$.

\begin{proposition}\label{prop:certificate}
Let $\Sigma_{\mathcal T}$ be the set of all closed pointed rank-five types
$\sigma$ (over all edge directions) having at least one facet in $\mathcal T$;
$|\Sigma_{\mathcal T}|=87{,}127$, over $551$ directions. There are rational
vectors $y_\tau\in\Q^{2}$, $\tau\in\mathcal T$, with denominators at most
$10^{4}$, such that for every $\sigma\in\Sigma_{\mathcal T}$
\[
  \alpha(\cone\sigma)+\sum_{\tau\in\mathcal T\text{ facet of }\sigma}
  \langle y_\tau,u_{\sigma/\tau}\rangle\ \ge\ 1.20\times10^{-3}.
\]
\end{proposition}

The vectors were found by maximizing the minimum slack of the displayed system
(a linear program with $87{,}127$ inequalities in $13{,}806$ unknowns; the
optimum is $1/768$), rounding the solution to rational numbers, and re-checking
every inequality exactly over $\Q$; the exact minimum is
\[
  \frac{53920464734520931563295528363}{44832837720505982456018162718720}
  \approx 1.2027\times10^{-3},
\]
and there is no violation. The vectors are dense ($6{,}902$ of the $6{,}903$
are nonzero).
Here $u_{\sigma/\tau}$ is computed from a $\Z$-basis $(f_1,f_2)$ of
$\operatorname{lin}(\tau)^{\perp}\cap\Z^{6}$ (so that $x\mapsto(f_1\cdot x,f_2\cdot x)$
is a surjection $\Z^{6}\to\Z^{2}$ with kernel $N_\tau$) as the common primitive
image of the extreme rays of $\sigma$ outside $\tau$.

\begin{proof}[Proof of Theorem~\ref{thm:linear}]
Let $H$ be a full-dimensional five-part hive and, for each two-face $F$, let
$y_{N(H,F)}$ be the vector of Proposition~\ref{prop:certificate} if the type of
$F$ lies in $\mathcal T$ and $0$ otherwise. By \eqref{eq:corrected} it suffices
to show that every bracket is nonnegative. Let $E$ be an edge of type $\sigma$.
If some facet of $\sigma$ lies in $\mathcal T$ then $\sigma\in\Sigma_{\mathcal T}$
and the bracket is at least $1.20\times10^{-3}$ by
Proposition~\ref{prop:certificate}. Otherwise the bracket equals
$\alpha(\cone\sigma)$. If the pulling triangulation of $\sigma$ contains a
negative cell then $\sigma\in\Sigma^{-}$; it cannot lie in $\Sigma^{-}_{<0}$,
because every facet of such a cone is in $\mathcal T$, so
$\alpha(\cone\sigma)\ge0$. If it contains no negative cell then
$\alpha(\cone\sigma)$ is a sum of positive cell weights. Hence $a_1(H)\ge0$.
Rational hives are covered by Lemma~\ref{lem:dilation}.
\end{proof}

\subsection{Replay}
The certificate is replayed by an independent checker that does not share the
enumeration code: for each of the $87{,}127$ types it re-derives the facets by
a direct enumeration of hyperplanes, checks closedness, pointedness and rank,
recomputes $\alpha$ as the sum of cell weights over a triangulation pulled from
the \emph{last} extreme ray (the enumeration used the first), checks that every
listed facet has exactly the claimed atlas set, verifies that $(f_1,f_2)$ is
integral, orthogonal to $\tau$ and has coprime $2\times2$ minors (hence is a
$\Z$-basis of $\operatorname{lin}(\tau)^{\perp}\cap\Z^{6}$), recomputes $u$,
and evaluates the inequality exactly; it also checks that every cone of
$\Sigma^{-}_{<0}$ is among the constraint types. All checks pass. A separate
data-driven check confirmed that every actual edge cone of the $300$ pool hives
with a facet in $\mathcal T$ is one of the $87{,}127$ constraint types, and that
each of the $15$ actual negative edge cones lies in $\Sigma^{-}$.

\section{The general frontier}\label{sec:frontier}

For a period-one hive polytope the local formula writes every Ehrhart coefficient
as a pairing of nonnegative face volumes with Berline--Vergne/Todd weights on
transverse cones. Thus Conjecture~\ref{conj:ktt} follows from a single
hive-specific effectivity statement: for every rank, codimension $q$, and
$2$-connected primitive-interior closed flat-rhombus coarsening $\Sigma$ with
weight vector $a_q$ and Minkowski boundary map $\partial_q$,
\begin{equation}\label{eq:hte}
  \exists\, y:\quad a_q+\partial_q^{\mathsf T}y\ \ge\ 0. \tag{HTE}
\end{equation}
By rational Farkas duality \eqref{eq:hte} is equivalent to
$\langle a_q,w\rangle\ge 0$ for every nonnegative balanced realizable face weight
$w$; pairing it with the (positive) face-volume weight of any hive polytope makes
every coefficient nonnegative. Proposition~\ref{prop:certificate} is precisely a
solution $y$ of \eqref{eq:hte} for $q=5$ in rank five (at the level of types,
with $\partial_5$ the two-face closure \eqref{eq:polygon}), and
Proposition~\ref{prop:negedge} shows that the correction $y$ cannot be dispensed
with: from codimension five on, even the closed transverse cones of actual hives
can have negative weight, so a general proof must be a global effectivity
argument rather than a pointwise one. Horn factorization \cite{KTT2006,KTTfactor}
disposes of separating boundary coarsenings, so the load-bearing case is the
$2$-connected primitive interior. We have neither a proof of \eqref{eq:hte} in
general nor a negative realizable balanced weight refuting it, and the
certificate found here is dense, suggesting no closed form to conjecture.

On the disproof side, one transfer route deserves record because it is
\emph{valid} and still fails. Realizing three disjoint horizontal strips as a
skew shape identifies the lattice points of any $3\times N$ transportation
polytope with Littlewood--Richardson tableaux of a fixed triple, compatibly
with dilation: for example, the polytope with row margins $(3,3,3)$ and column
margins $(2,1^{7})$ has $L_T(n)=c^{\,n\nu}_{n\lam,\,n\mu}$ for
$\lam=(9,6,3)$, $\mu=(9,9,7,6,5,4,3,2,1)$, $\nu=(15,12,9,7,6,5,4,3,2,1)$
(both counts equal $1050$ at $n=1$). Hence a $3\times 8$ transportation
polytope with a negative Ehrhart coefficient would refute
Conjecture~\ref{conj:ktt} outright. None exists in the relevant regime: over
the full dimension-$14$, codegree-$3$ margin family the minimum linear
coefficient is $2157/280>0$ (exhaustive for total margin $\le 19$, with
saturation in every unbounded direction, and with the closed form
$\tfrac{3a}{2}\bigl(1+H_{2a-1}-H_a\bigr)>0$ for the symmetric unit-column
subfamily); moreover codegree $3$ forces $L_T(1)\ge 1050$, so the known
negative order polytope of Liu and Tsuchiya \cite{LiuTsuchiya}, which has
$255$ lattice points, is not realizable in this shape. The route evades the
positivity theorems for skew Gelfand--Tsetlin polytopes (its content
constraints are global sums, not markings) and dies only on these measured
facts.

We record, finally, that the standard rank-uniform shortcuts are individually
obstructed by exact examples: generic type-$A$ Todd/cocircuit balancing (a
nonnegative $A_4$ Dahmen--Micchelli polynomial maps to
$N^4+4N^3+3N^2-2N-\tfrac65$); network-flow equivalence (simple hive tangent cones
have determinant two); generic Cohen--Macaulay/Koszul Hilbert-ring structure (the
$Q_{20}$ Hibi ring is Gorenstein, Koszul, with negative $a$-invariant, yet its
Hilbert polynomial has linear coefficient $-168011/330$); matroidal/secondary-fan
Todd effectivity (a genuine hive face carries a $U_{2,4}$ tight-normal
restriction); positive tableau/Kostka recursions (they require signs, or stop on
primitive triples); and rank-by-rank Hurwitz stability (an unbounded hierarchy).
Six or more parts remain open below the top four coefficients; repeating the
argument of Section~\ref{sec:linear} in rank six would require corrections in
five codimensions simultaneously, which is the unbounded certificate cascade
that \eqref{eq:hte} is meant to replace. Details and the supporting computations
are in the accompanying reports.

\section{Reproducibility}\label{sec:repro}

Every finite claim above is verified by exact arithmetic and reproduced by
scripts distributed with this paper.
\begin{itemize}
\item Two independent Littlewood--Richardson counters (a hive lattice-point
enumerator and a lattice-word LR-tableau counter), cross-checked against each
other on all triples with $|\nu|\le 8$, against $c=1,2$ stretched values, and
against \cite{lrcalc}.
\item The rank-four atlas of $15$ normals and $99$ pairs, and
Proposition~\ref{prop:codim2} (all Berline--Vergne weights, minimum $1/9$).
\item An independent Ehrhart interpolation that recomputes $L_H(0),\dots,L_H(D+2)$
and confirms the interpolated polynomial at two held-out points, checked against
both LR counters; and a verification of $a_1=\sum_E\alpha\,\ell_{\Z}(E)$ against
lattice-count $a_1$ on more than $10^4$ four-part hives that are not among the
atlas basis polytopes.
\item The rank-uniform codimension-two and codimension-three enumerations, and
the full codimension-four five-part enumeration with the closure repairs.
\item For Section~\ref{sec:linear}: the cell census and the $887$-hive
validation (\texttt{r5\_codim5\_edge\_census.py}), the enumeration of
$\Sigma^{-}$ (\texttt{r5\_codim5\_coverage.py}), the enumeration of
$\Sigma_{\mathcal T}$ and the linear program (\texttt{r5\_codim5\_farkas.py}),
the certificate (\texttt{r5\_codim5\_farkas\_certificate.json}) and the
independent checker (\texttt{r5\_codim5\_verify\_certificate.py}). The
closure steps of the enumerations use the exact rational double-description
code of \texttt{cddlib} \cite{cddlib}; the checker does not. The linear
program was solved with HiGHS \cite{HiGHS}, which only proposes the
certificate; every verdict is an exact rational computation.
\item For Section~\ref{sec:transfer}: the symbolic check of the strict hive and
the finite consistency check of the transfer mechanism
(\texttt{chamber\_boundary\_transfer\_check.py}).
\end{itemize}
The dilation identity of Lemma~\ref{lem:dilation} and the nonunimodular vertex of
Section~\ref{sec:nonsimple} are likewise reproduced exactly.

The complete reproducibility bundle---the two Littlewood--Richardson counters,
the rank-four atlas and codimension-two/three enumerations, the five-part
codimension-four certificate, the five-part linear-coefficient certificate with
its checker, the transfer check, and the exact replay scripts for every
displayed value---is available at
\url{https://github.com/AlperTheKing/ktt-positivity}.

\end{document}